\documentclass[11pt]{amsart}
\usepackage{amssymb,amsmath,amsfonts,latexsym,amsthm,paralist,graphicx}

\newtheorem{theorem}{Theorem}[section]
\newtheorem{corollary}[theorem]{Corollary}

\theoremstyle{definition}

\newtheorem{remark}[theorem]{Remark}
\newtheorem{definition}[theorem]{Definition}
\DeclareMathOperator{\Dens}{\overline{Dens}}
\DeclareMathOperator{\spa}{span}

\input epsf

\begin{document}

\title{Linear chaos for the Quick-Thinking-Driver model}

\date{}

\author[Conejero]{J. A. Conejero}
\address{Instituto Universitario de Matem\'{a}tica Pura y Aplicada,\newline\indent Universitat Polit\`{e}cnica de Val\`{e}ncia, \newline\indent 46022, Val\`{e}ncia, Spain.}
\email{aconejero@upv.es}

\author[Murillo]{M. Murillo Arcila}
\address{Instituto Universitario de Matem\'{a}tica Pura y Aplicada,\newline\indent Universitat Polit\`{e}cnica de Val\`{e}ncia, \newline\indent 46022, Val\`{e}ncia, Spain.}
\email{mamuar1@posgrado.upv.es}

\author[Seoane]{J.B. Seoane-Sep\'{u}lveda}
\address{Departamento de An\'{a}lisis Matem\'{a}tico,\newline\indent Facultad de Ciencias Matem\'{a}ticas, \newline\indent Plaza de Ciencias 3, \newline\indent Universidad Complutense de Madrid,\newline\indent Madrid, 28040, Spain.}
\email{jseoane@mat.ucm.es}


\thanks{MEC Project MTM2013-47093-P and by MTM2012-34341.}

\begin{abstract}
In recent years, the topic of car-following has experimented an increased importance in traffic engineering and safety research. This has become a very interesting topic because of the development of driverless cars \cite{google_driverless_cars}.

Driving models which describe the interaction between adjacent vehicles in the same lane have a big interest in simulation modeling, such as the Quick-Thinking-Driver model. A non-linear version of it can be given using the logistic map, and then chaos appears.
We show that an infinite-dimensional version of the linear model presents a chaotic behaviour using the same approach as for studying chaos of death models of cell growth.
\end{abstract}

\maketitle

\section{Introduction}
\label{intro}

Dynamical systems on infinite-dimensional linear spaces allow us to describe phenomena in which a chaotic behaviour is exhibited.
The dynamics of the solution $C_0$-semigroups of gene amplification--deamplification processes with cell proliferation
have been widely studied by Banasiak et al. \cite{banasiak_lachowicz2001chaos,banasiak_lachowicz2002topological,banasiak_lachowicz_moszynski2003topological,banasiak_moszynski2011dynamics} (see, also, \cite{aroza_peris2012chaotic,delaubenfels_emamirad_protopopescu2000linear,grosse-erdmann_peris-manguillot2011linear}).
Such a behaviour can also be found on certain size structured cell populations, c.f. \cite{el-mourchid_metafune_rhandi_voigt2008on,el-mourchid_rhandi_vogt_voigt2009a}.

These models are very similar to the ones used to describe traffic. Models of car-following describe the interaction between adjacent vehicles in the same lane and show how one vehicle follows one another. Since the  studies of Greenshields \cite{greenshields1934the,greenshields1935a} in the 1930's, several models started to appear sharing this point.
See for instance the work of Pipes \cite{pipes1953an}, Helly \cite{helly1953simulation}, and Herman, Montroll et al \cite{chandler_herman_montroll1958traffic,herman_montroll_elliott_potts_rothery1959traffic} in the 1950's.
These models were perefectioned and improved. For a historical evolution of these models we refer the interested reader to \cite{brackstone_mcdonald1999car-following}.

During the 2000's, great advances in robotics and artificial intelligence have permited to develop prototypes of autonomous cars. The advances can be followed looking at the DARPA Grand Challenge competition \cite{darpa_challenge_competition}. In fact, motor companies like Audi or BMW have presented their first models very recently \cite{bmw_audi}.

We would like to emphasize that these seminal models, which are based on simple linear equations, cannot describe all the highly complex situations that, actually, occur on a roadway. These theoretical models work better on long stretches of road with dense traffic. In this note we consider one of these models, the Quick-Thinking-Driver model, for an infinite number of cars on a track. The infinite system of ordinary differential equations that are required to model the behaviour of these vehicles can be represented as a linear operator on an infinite-dimensional separable Banach space. Then, using some results of linear dynamics of $C_0$-semigroups we can prove the existence of different chaotic behaviours for the solutions of these equations.

\section{Preliminaries}

In the basic formulation of any of the car-following models, there is a relation between the acceleration of a car and the difference between its velocity and the velocity of the car that goes in front of it. We assume that a driver adjusts her speed according to her relative speed respect to that car. This can be described by means of the following equation
\begin{equation}\label{generalmodel}
u_1'(t+t_1)=\lambda_1(u_{2}(t)-u_1(t)),
\end{equation}
in which $u_1$ stands for the velocity of the car $1$, $u_2$ for the velocity of the car in front of the car $1$, namely car $2$, $t_1$ denotes the reaction time of driver $1$, and the positive number $\lambda_1$ is a sensitivity coefficient that measures how strong the driver $1$ responds to the acceleration of the car in front of her. Usually $\lambda_1$ lies
between $0.3-0.4s^{-1}$ \cite{brackstone_mcdonald1999car-following}.
Under the assumption that all drivers react ``very quickly'', one can take $t_1=0$.
This is known as the \emph{Quick-Thinking-Driver} (QTD) model. 

If we consider this model for an infinite number of cars circulating on a road, then the QTD model is given by an infinite system of first-order ordinary differential equations. As far as we know, this has not been considered yet in the literature:

\begin{definition}[The Infinite Quick-Thinking-Driver model]
	Let $(u_i(t))_i$ denote the velocities of an infinite number of cars on a track. Suppose that for every $i\in\mathbb{N}$,  the car $i$ behaves following the (QTD) model respect to one in front of it, the car $i+1$. The following infinite system of ordinary differential equations represents this situation.
	\begin{equation}\label{quickmodel}
	u_i'(t)=\lambda_i(u_{i+1}(t)-u_i(t))\text{ for }i\in\mathbb{N}.
	\end{equation}
	
	We also assume that the velocities at $t=0$, $(u_i(0))_i$, are given.
	We call to this model the \emph{Infinite Quick-Thinking-Driver model}. We will refer to it as the (IQTD) model. 
\end{definition}

McCartney and Gibson studied chaos for the (QTD) model making the parameter $\lambda_i$  to be proportional to the velocity \eqref{quickmodel}, i.e. $\lambda_i$ is replaced by $\gamma u_i(t)$, with $\gamma$ being the new sensitivity coefficient, and making the velocity of the leading car to be constant. This converts  \eqref{generalmodel} into the logistic equation, see also \cite{li2005nonlinear,lo_cho2005chaos} and then chaotic phenomena can be found in its solutions. We will see that  even in the linear (IQTD) model given by \eqref{quickmodel} chaos still appears.

Let $X$ be a separable infinite-dimensional Banach space. 
We assume that the reader is familiar with the terminology of $C_0$-semigroups on Banach spaces, see for instance \cite{engel_nagel2000one-parameter}. 

Next, we recall some basic definitions on linear dynamics of $C_0$-semigroups.
A $C_0$-semigroup $\{T_t\}_{t\geq 0}$ on $X$ is said to be {\em hypercyclic} if there exists $x\in X$ such that the set $\{T_tx:t\geq 0\}$ is dense in $X$. An element $x\in X$ is  a \emph{periodic point} for the semigroup if there exists $t>0$ such that $T_tx=x$. A semigroup $\{T_t\}_{t\geq 0}$ is called \emph{Devaney chaotic} if it is hypercyclic and the set of periodic points is dense in $X$.
In this setting, Devaney chaos yields the sensitive dependence on the initial conditions, as it was seen by Banks et al  \cite{banks_brooks_cairns_davis_stacey1992on,grosse-erdmann_peris-manguillot2011linear}.

Another extended definition of chaos is \emph{distributional chaos} \cite{martinez-gimenez_oprocha_peris2009distributional}. A $C_0$-semigroup $\{T_t\}_{t\geq 0}$ on $X$ is said to be \emph{distributionally chaotic} if there exists an uncountable subset $S\subset X$ and $\delta>0$ such that, for each pair of distinct points $x,y\in S$ and for every $\epsilon>0$, we have $\Dens(\{s\geq 0;||T_sx-T_sy||>\delta\})=1$ and $\Dens(\{s\geq 0;||T_sx-T_sy||<\epsilon\})=1$, where the upper density of a set $A\subset\mathbb{R}$ is defined as $$\Dens(A)=\limsup_{t\rightarrow\infty}\frac{\mu(A\cap[0,t])}{t},$$
where $\mu$ denotes the Lebesgue measure on $\mathbb{R}$.
We will show that both Devaney and distributional chaos appear in the solutions of the (IQTD) model using a suitable underlying Banach space.

\section{Chaos for the Quick-Thinking-Driver model}
Let us consider $\ell_1(s)$, $0<s\le1$, the weighted space of summable sequences defined as
\begin{equation}
\ell_1(s)=\left\{(v_i)_{i\in\mathbb{N}}\in\mathbb{K}^\mathbb{N}:||(v_i)_{i\in\mathbb{N}}||_s=\sum_{i\in\mathbb{N}}|v_i|s^i<\infty\right\}
\end{equation} 

If $s<1$, then we can consider the velocities of all the cars of the (IQTD) model in $\ell_1(s)$. 
We point out that such a choice of weights gives more importance to velocities of cars with low index $i$.

In order to solve the system of equations in \eqref{quickmodel}, we pose the following abstract Cauchy problem on $\ell_1(s)$
\begin{equation}\label{quickinfinite}
\left\{\begin{array}{rl}
u'(t )&=  Au(t)  \\
u(0) & =(u_i(0))_{i\in\mathbb{N}}
\end{array}\right\},
\end{equation}
where the operator $A$ is defined as 
\begin{equation}\label{equation_A}
(Au(t))_i=\lambda_i(u_{i+1}(t)-u_i(t)), \quad i\in\mathbb{N}
\end{equation}
for $u=(u_i(t))_{i\in\mathbb{N}}\in \ell_1(s)$ and $t\ge 0$, being $(u_i(0))_{i\in\mathbb{N}}$ the vector of velocities of the cars at $t=0$.


If the $\sup_{i\in\mathbb{N}}|\lambda_i|<\infty$, then the solution to \eqref{quickinfinite} can be represented by a $C_0$-semigroup $\{T_t\}_{t\ge 0}$ on $\ell^1(s)$ whose infinitesimal generator is $A$ with domain the whole space $\ell^1(s)$ . Then the operators in this $C_0$-semigroup can be represented as $T_t=e^{tA}=\sum_{k=0}^\infty (tA)^n/n!$ for all $t\ge 0$, see for instance \cite{engel_nagel2000one-parameter}*{Ch. I, Prop. 3.5}.

When considering in $\ell_1(1)$, the death models used for modeling cell-growth are of the form
\begin{equation}\label{death_model}
u_i'(t)=\beta_iu_{i+1}(t)-\alpha_i u_i(t) \quad\text{ with }\; \alpha_i,\beta_i>0\; \text{ for }i\in\mathbb{N}
\end{equation}

In proposition 7.36 of  \cite{grosse-erdmann_peris-manguillot2011linear}, they proved the following result:
\begin{theorem}
	If $\sup_i\alpha_i<\liminf_{i\rightarrow \infty}\beta_i$, then the solution $C_0$-semigroup of \eqref{death_model} is Devaney chaotic (and topologically mixing) on $\ell_1(1)$.
\end{theorem}

More details can be found in \cite{aroza_peris2012chaotic,banasiak_lachowicz2001chaos}.
This proof is based on an application of the Desch-Schappacher-Webb (DSW) criterion \cite{desch_schappacher_webb1997hypercyclic}, see also \cite{banasiak_moszynski2005a,el-mourchid2006the,grosse-erdmann_peris-manguillot2011linear} for other reformulations of this result.
The application of the (DSW) criterion relies on verifying that the point spectrum of the infinitesimal generator of the $C_0$-semigroup contains ``enough'' eigenvalues. A criterion directly stated in these terms was firstly given for operators by Godefroy and Shapiro in \cite{godefroy_shapiro1991operators}. We will use the following version of the (DSW) Criterion, see \cite{grosse-erdmann_peris-manguillot2011linear}*{Th. 7.30}.

\begin{theorem}\label{dsw_gepm_eigenvaluecriterion}
	Let $X$ be a complex separable Banach space, and $\{T_t\}_{t\ge 0}$ a $C_0$-semigroup on $X$ with infinitesimal generator $(A,D(A))$. Assume that there exists an open connected subset $U$ and a weakly holomorphic function $f:U\rightarrow X$, such that
	\begin{enumerate} 
		\item[(1)] $U\cap i\mathbb{R}\ne \emptyset$,
		\item[(2)] $f(\lambda)\in\ker (\lambda I-A)$ for every $\lambda\in U$,
		\item[(3)] for any $x^*\in X^*$, if $\langle f(\lambda),x^*\rangle = 0$ for all $\lambda\in U$, then $x^*=0$.
	\end{enumerate}
	Then the semigroup $\{T_t\}_{t\ge 0}$ is Devaney chaotic (and topologically mixing).
\end{theorem}


Since we can assume $0.3<\lambda_i<0.4$ for all $i\in\mathbb{N}$, c.f.  \cite{brackstone_mcdonald1999car-following}, we can apply this criterion to the (IQTD) model. The proof of the following result can be compared with the one of \cite{grosse-erdmann_peris-manguillot2011linear}*{Prop. 7.36}.

\begin{theorem}Let $0<\alpha<\beta$. 
	Let us also consider the (IQTD) model on $\ell_1(s)$, with $0<s<\alpha/\beta$. If  $\alpha\le\lambda_i\le\beta$ for all $i\in\mathbb{N}$, then the solution $C_0$-semigroup of the (IQTD) model is Devaney chaotic on $\ell_1(s)$.
\end{theorem}
\begin{proof}
	Fix $s$ with $0<s<\alpha/\beta$ and consider the solution $C_0$-semigroup of the (IQTD) model whose infinitesimal generator is given by the operator $A$ defined in \eqref{equation_A}. 
	For applying Theorem \ref{dsw_gepm_eigenvaluecriterion}, we consider as $U$ the open disk centered at $0$ and radius $0<\varepsilon$ such that $(\varepsilon+\beta)s/\alpha<1$, i.e. $U=\{\mu\in\mathbb{C}\,:\,|\mu|<\varepsilon\}$, and then (1) of Theorem \ref{dsw_gepm_eigenvaluecriterion} holds. We define the function $f:U\rightarrow \ell_1(s)$ as $f(\mu)=h_\mu$,  with $h_\mu$ defined as  
	\begin{equation}
	h_\mu=\left(1,\frac{(\mu+\lambda_1)}{\lambda_1},\frac{(\mu+\lambda_1)}{{\lambda_1}}\frac{(\mu+\lambda_2)}{\lambda_2},\ldots,\prod_{j=1}^{i-1}\frac{\mu+\lambda_j}{\lambda_j},\ldots\right).
	\end{equation}
	
	Then $f(\mu)\in \ell_1(s)$ because, by the choice of $\mu$, we have
	\begin{align}\label {cuentecita} 
	||h_\mu||_s  &=\sum_{i=1}^\infty|(h_\mu)_i|s^i<s+\sum_{i=2}^\infty\prod_{j=1}^{i-1}\frac{|\mu+\lambda_j|}{|\lambda_j|}s^{i}
	<s+\sum_{i=2}^\infty\prod_{j=1}^{i-1}\frac{|\mu|+|\lambda_j|}{|\lambda_j|}s^{i}\\
	&<s+\sum_{i=2}^\infty s \left(\frac{(\varepsilon+\beta)s}{\alpha}\right)^{i-1}<\infty.
	\end{align}
	
	In fact, $f(\mu)$ is an eigenvector of $A$ associated to the eigenvalue $\mu$. This gives (2) of Theorem \ref{dsw_gepm_eigenvaluecriterion}.
	Finally, for proving (3) of Theorem \ref{dsw_gepm_eigenvaluecriterion}, let $\gamma\in \ell_\infty(1/s)$, the dual space of $\ell_1(s)$, such that $\langle f(\mu),\gamma\rangle=0$ for all $\lambda\in U$. The function $\mu\rightarrow  \langle f(\mu),\gamma\rangle$ defined as
	\begin{equation}\label{elemento_dual}
	\langle f(\mu),\gamma\rangle=\gamma_1+\sum_{i=2}^\infty \gamma_{i}\left(\prod_{j=1}^{i-1}\frac{\mu+\lambda_j}{\lambda_j}\right)
	\end{equation}
	converges uniformly for all $\mu\in U$ by \ref{cuentecita}, which gives that $f$ is weakly holomorphic on $U$.
	
	If we consider $\mu=-\lambda_1\in U$ we have that \eqref{elemento_dual} is reduce to $\gamma_1$, which gives $\gamma_1=0$.  Then, we divide the expression in \eqref{elemento_dual} by $\mu+\lambda_1$. The result is still a weakly holomorphic function on $U$. Now, if we consider $\mu=-\lambda_2\in U$, then we get $\gamma_2=0$. Repeating the process inductively, we get that all $\gamma_i=0$, $i\in\mathbb{N}$, and the proof is finished.
	
	%
	%
	
\end{proof}

Clearly, we get Devaney chaos for the aforementioned experimental values obtained for the (IQYD) model.

\begin{corollary}
	Let us consider the (IQTD) model on $\ell_1(s)$, with $0<s<3/4$. If $0.3\le\lambda_i\le0.4$ for all $i\in\mathbb{N}$, then the solution $C_0$-semigroup of the (IQTD) model is Devaney chaotic on $\ell_1(s)$.
\end{corollary} 

As a result, if an infinite numbers of cars drive along a road and all the drivers respond in the same way to any change in the velocity of the car in front, this may lead to a chaotic situation, depending on the initial velocities of the cars. 
Therefore other considerations should be taken into account in order to manage the speed adjustment.
This also contrasts with the behaviour of the $C_0$-semigroup  solution on $\ell_1$  studied in \cite{delaubenfels_emamirad_protopopescu2000linear} where it is proved that this model is not Devaney chaotic on $\ell_1(1)$.

Up to here, we have studied Devaney chaos for the (IQTD) model. 
It is also well known that distributional chaos holds whenever the DSW criterion can be applied \cite{barrachina_conejero2012devaney,bermudez_bonilla_martinez-gimenez_peris2011li-yorke}. As a result, we obtain the following corollary.
\begin{corollary}
	Let us consider the (IQTD) model on $\ell_1(s)$, with $0<s<3/4$. If $0.3<\lambda_i<0.4$ for all $i\in\mathbb{N}$, then the solution $C_0$-semigroup of the (IQTD) model is distributionally chaotic on $\ell_1(s)$.
\end{corollary}

Roughly speaking, this means that we can pick two vectors of initial velocities from an uncountable set such that there will be intervals of arbitrary time length in which the velocities of the cars are very similar and  intervals in which there exists at least some positive difference between their velocities.

We conclude this section with a comment regarding the stability of the solutions.
An analysis of  local and asymptotic stability of the (QTD) model was considered in \cite{herman_montroll_elliott_potts_rothery1959traffic}. From the point of view of $C_0$-semigroups, 
we recall that a $C_0$-semigroup of the form
$\{e^{tA}\}_{t\ge 0}$ defined on a Banach space $X$ is \emph{(uniformly) exponentially stable}, \cite{engel_nagel2000one-parameter},
if there exists $\eta>0$ such that
\begin{equation}
\lim_{t\rightarrow \infty} e^{\eta t}||e^{tA}||=0.
\end{equation}

This situation is not fulfilled in our case. Nevertheless, a weaker version of this condition can also be considered.
We say that $\{e^{tA}\}_{t\ge 0}$ is \emph{exponentially stable on a subspace} $Y\subset X$
if there exists $\eta>0$ such that for any $y\in Y$ we have
\begin{equation}
\lim_{t\rightarrow \infty} e^{\eta t}||e^{tA}y||=0,
\end{equation}

This analysis is sometimes considered when studying chaos of $C_0$-semigroups,
as it is done in \cite{banasiak_moszynski2011dynamics,brzezniak_dawidowicz2009on,conejero_rodenas_trujillo2015chaos}.

\begin{theorem}\label{stability_IQTD}
	Let $0<\alpha<\beta$. Let us consider the (IQTD) model on $\ell_1(s)$, with $0<s<\alpha/\beta$. If $\alpha\le\lambda_i\le\beta$ for all $i\in\mathbb{N}$, then the solution $C_0$-semigroup of the $(IQTD)$ model is exponentially stable on the subspaces
	\begin{equation}
	Y_\delta:=\spa\{h_\mu \,:\, |\mu|<\varepsilon,\, \Re(\mu)<\delta\},
	\end{equation}
	for every $-\varepsilon<\delta<0$, where $\varepsilon>0$ satisfies $s(\varepsilon+\beta)/\alpha<1$.
\end{theorem}
\begin{proof}
	Fix $0<s<\alpha/\beta$, $\varepsilon>0$ satisfying $s(\varepsilon+\beta)/\alpha<1$
	and $-\varepsilon<\delta<0$. Take $0<\eta<-\delta$ and $y\in Y_\delta$ of the form $y=\sum_{i=1}^k\alpha_i h_{\mu_i}$. Define $\delta_y=\max\{\Re(\mu_i)\,:\,1\le i\le k\}$. Clearly, $\eta+\delta_y<0$, then
	\begin{align*}
	e^{\eta t}||e^{tA}y||_s &=e^{\eta t}\left|\left|\sum_{i=1}^k\alpha_ie^{t\mu_i}h_{\mu_i}\right|\right|_s\le e^{\eta t} \left(\sum_{i=1}^ke^{t\Re(\mu_i)}||\alpha_ih_{\mu_i}||_s\right)\\
	& < e^{t(\eta+\delta_y)}\left(\sum_{i=1}^k||\alpha_i h_{\mu_i}||_s\right)
	\end{align*}
	which tends to $0$ when $t$ tends to $\infty$.
\end{proof}

\begin{remark}
	For the experimental values of $0.3\le\lambda_i\le 0.4$ for all $i\in\mathbb{N}$, Theorem \ref{stability_IQTD} states that there are subspaces where the solution $C_0$-semigroup of the $(IQTD)$ model is exponentially stable on $\ell_1(s)$, for $0<s<3/4$. Moreover, in these spaces there are also subspaces where the solution $C_0$-semigroup of the $(IQTD)$ model is not exponentially stable.
	If $\varepsilon>0$ satisfies $s(\varepsilon+\beta)/\alpha<1$, we can take, for instance, the case where initial conditions are taken on 
	$Z:=\spa\{h_\mu \,:\, |\mu|<\varepsilon,\, \Re(\mu)>0\}$,  since we have that $\lim_{t\rightarrow \infty} ||T_tz||_s=\infty$ for all $z\in Z\setminus\{0\}$.
\end{remark}

\section{Final comments}
In this note we have related the dynamics of models of traffic following with the dynamics of death process of cell growth. We also point out that some other models like the one of forward and backward control \cite{herman_montroll_elliott_potts_rothery1959traffic} can also be related with birth-and-death process of 
cell--proliferation.

It is fair to note that the chaotic properties reported in this paper refer to the whole-space of solutions whereas  only positive solutions make sense. Therefore,  being important, e.g., in the analysis of the stability of numerical schemes, does not necessarily say anything important about the actual behaviour of the ``physical'' trajectories of the cars, in the same way as it is commented in \cite{banasiak_lachowicz_moszynski2003topological}. To sum up, in dense traffic drivers follow one another very closely and small disturbances such as acceleration or deceleration of one vehicle might be passed over or amplified along the line of vehicles on the road.
These disturbances can be considered in \eqref{quickinfinite} by adding to the right part of the equations a term related to the notion of acceleration noise (a kind of measure for the smoothness or the quality of traffic flow, see \cite{accnoise}).
A chaotic flow appears and this can lead to an impredictible dynamical behaviour and, in some cases, to accidents.

\bibliographystyle{spmpsci}      

\end{document}